\documentclass[12pt]{amsart}
\usepackage{amssymb,enumitem,amsthm,amsmath,bibentry,thmtools,ulem,mathtools,xurl,microtype}
\usepackage[pdftex,colorlinks=true,linkcolor=Black]{hyperref}
\usepackage[usenames,dvipsnames]{xcolor}

\numberwithin{equation}{section}

\newcommand{\la}{{\lambda}}

\newcommand{\ddt}{{\frac{d}{dt}}}

\newcommand{\norm}[1]{\left\lVert#1\right\rVert}
\newcommand{\Lnorm}[2]{\left\lVert#1\right\rVert_{L^#2}}

\newcommand{\be}{\begin{equation}}
\newcommand{\ee}{\end{equation}}
\newcommand{\ba}{\begin{aligned}}
\newcommand{\ea}{\end{aligned} }

\theoremstyle{plain}
\newtheorem{theorem}{Theorem}[section]
\newtheorem{lemma}[theorem]{Lemma}

\theoremstyle{definition}

\theoremstyle{remark}
\newtheorem{remark}[theorem]{Remark}

\begin{document}

\title[Global Regularity for Axisymmetric NS with NHL BC]{Global Regularity of Axisymmetric Navier--Stokes Equations with NHL Boundary Condition under a Critical Smallness Condition
}

\thanks{}

\author[T. L. Chan]{Tsz-Lik Chan}

\address[]{Department of mathematics, University of California, Riverside, CA 92521,
USA}

\email{tchan204@ucr.edu}



\begin{abstract}
We investigate the global regularity problem for the three-dimensional incompressible Navier--Stokes equations restricted to axisymmetric flows in a finite cylinder $D = \{(r,\theta,x_3): 0 \le r \le 1, 0 \leq \theta < 2\pi,\ 0 \le x_3 \le 1\}$, subject to the Navier--Hodge--Lions (NHL) boundary condition. While global existence of smooth solutions is known in the swirl-free case, the presence of swirl ($v_\theta \neq 0$) introduces vortex stretching that may potentially lead to finite-time singularity formation. 

In this work, we prove that if the initial data satisfy a scaling-invariant smallness condition of the form
\[
\frac{9C_1C_3^{1/2}}{4}\left(\frac{1}{2}\|V_0\|_{L^4}^4 + \|\Omega_0\|_{L^2}^2\right)^{1/4}\|\Gamma_0\|_{L^4} \le \frac{1}{4},
\]
where $V = v_\theta/\sqrt{r}$, $\Omega = \omega_\theta/r$, $\Gamma = r v_\theta$, and $C_1, C_3$ are explicit constants given in this paper, then the solution remains globally regular for all time. The proof proceeds via a transformed system for $\Omega$ and $V$, leveraging a maximum principle for $\Gamma$, refined Agmon-type inequalities to control $\|v_r/r\|_{L^\infty}$, and delicate boundary analysis of the finite cylinder geometry. Key energy estimates yield $L^\infty_T L^4_x$ bounds for all velocity components, which fall within the regularity class, thereby precluding finite-time blow-up. The result extends the known criticality theory for axisymmetric Navier--Stokes flows to the setting of NHL boundary conditions, which are physically relevant for flows with stress-free or slip-type constraints on lateral and horizontal boundaries.
\end{abstract}
\maketitle

\section{Introduction}
The incompressible Navier--Stokes equations are
\begin{equation} \begin{dcases}
    \label{eq:NS}
    &\partial_t v + (v \cdot \nabla) v = \nu \Delta v - \nabla p + f, \\
    & \nabla \cdot v = 0.
\end{dcases} \end{equation}
where \(v\) is the velocity field, \(p\) is the pressure, $f$ is the external force which will be taken as $0$ in the analysis and \(\nu > 0\) is the kinematic viscosity which will be taken as $1$ in the analysis.  

The incompressible Navier--Stokes equations model the motion of viscous fluids with constant density (e.g., water). The equations encode two fundamental physical principles: conservation of momentum and conservation of mass. 

The first equation in \eqref{eq:NS} gives the conservation of momentum, or Newton's second law. The meaning of each term in the first equation in \eqref{eq:NS} is:
\begin{itemize}
    \item $\partial_t v + (v \cdot \nabla) v $: total fluid acceleration (local change $+$ convective transport by flow),
    \item $-\nabla p$: pressure gradients driving fluid from high to low pressure,
    \item $\nu \Delta v$: viscous diffusion smoothing velocity gradients (internal friction),
    \item $f$: external forces acting throughout the fluid.
\end{itemize}

The second equation in \eqref{eq:NS} gives the conservation of mass. If the solution is regular, then the divergence-free condition is equivalent to saying that the flow map $X(t,a)$, which represents the position of particle initially at $a$ at time $t$, is volume-preserving. That is the measure $ |\{X(t,a): a \in V \}|$ for some given set $V$ is constant \cite{bedrossian2022mathematical}. 

The kinematic viscosity $\nu$ determines how ``sticky" a fluid is. The inverse of the viscosity is proportional to another important quantity in engineering and other fields called the \textit{Reynolds number}. If the Reynolds number is high then the more prominent term in the equation will be the momentum term. If the Reynolds number is low, then the ﬂuid will be dominated by viscosity.

The 2D global regularity problem was resolved by Leray for $\mathbb{R}^2$ in \cite{leray1933etude} and in bounded domains by Ladyzhenskaya in \cite{ladyzhenskaya1959solution}. For a bounded domain $\Omega \subset \mathbb{R}^2$ with smooth boundary, consider the initial-boundary value problem:
\begin{equation}
\begin{dcases}
\displaystyle \partial_t v + (v \cdot \nabla)v = -\nabla p + \nu \Delta v + f, & \text{in } \Omega \times (0,T], \\
\nabla \cdot v = 0, & \text{in } \Omega \times (0,T], \\
v = 0, & \text{on } \partial\Omega \times (0,T], \\
v(\cdot,0) = v_0, & \text{in } \Omega,
\end{dcases}
\end{equation}
where $f$ is an external force. Assuming $v_0 \in H^1_0(\Omega)$ with $\nabla \cdot v_0 = 0$ and $f \in L^2(0,T; L^2(\Omega))$, Ladyzhenskaya proved the existence of a unique global regular solution. Consequently, solutions remain smooth for any time $T > 0$ without finite-time blow-up or singularity formation. The results in \cite{leray1933etude,ladyzhenskaya1959solution} not only provided the rigorous foundation for 2D fluid modeling (critical for weather prediction, oceanography, and engineering simulations) but also introduced analytical techniques.

In the three-dimensional case, the concept of \textit{Leray--Hopf weak solutions} constitutes a cornerstone in the mathematical analysis of the incompressible Navier--Stokes equations. For a domain $\Omega \subseteq \mathbb{R}^3$ (either $\mathbb{R}^3$ itself or a bounded domain with smooth boundary), given an initial divergence-free velocity field $v_0 \in L^2_{\sigma}(\Omega)$ (where $L^2_{\sigma}$ denotes the closure of smooth divergence free fields in $L^2$) and kinematic viscosity $\nu > 0$, a Leray--Hopf weak solution is a vector field $v$ satisfying: (i) $v$ is in the energy space: $v \in L^\infty(0,T; L^2_{\sigma}(\Omega)) \cap L^2(0,T; H^1(\Omega))$ for all $T>0$; (ii) the Navier--Stokes equations hold in the distributional sense; (iii) $v$ converges to $v_0$ weakly in $L^2$ as $t \to 0^+$; and (iv) $v$ satisfies the \textit{energy inequality}
\[
E(v):= \|v(t)\|_{L^2(\Omega)}^2 + 2\nu \int_0^t \|\nabla v(s)\|_{L^2(\Omega)}^2  \mathrm{d}s \leq \|v_0\|_{L^2(\Omega)}^2 \quad \text{for a.e. } t \geq 0,
\]
which provides the essential a priori bounds. 

Historically, Leray in the seminal paper \cite{leray1934mouvement} first established global existence of weak solutions for the Cauchy problem in $\mathbb{R}^3$ and introduced the energy inequality—referring to such solutions as \textit{``solutions turbulentes''}. Hopf extended this framework to bounded domains with Dirichlet boundary conditions (\cite{hopf1950anfangswertaufgabe}), proving existence while preserving the energy inequality. The combined legacy of their foundational work—Leray's global existence theory and Hopf's treatment of boundary effects—led to the modern designation of \textit{Leray--Hopf weak solutions}, which remain central to existence theory, partial regularity results, and the ongoing study of uniqueness and smoothness in fluid dynamics.

Dimensionality fundamentally alters the behavior of nonlinear evolutionary systems since the two-dimensional regularity problem stands in sharp contrast to the three-dimensional case. While the existence of global Leray--Hopf weak solutions is known, whether such solutions remain smooth for all time, or whether finite-time singularities can develop from smooth initial data, remains unresolved.  

The natural scaling for the incompressible Navier--Stokes equations is given by $$v_\lambda = \lambda v(\lambda x, \lambda^2 t), \quad p_\lambda = \lambda^2 p(\lambda x, \lambda^2 t),$$ which preserves the structure of the equations. Under this scaling, the energy norm $E(\cdot)$ scales as $$E(v_\lambda) = \la^{2-n}E(v).$$ Consequently, the energy is invariant under scaling (critical) precisely when $n=2$. For $n>2$ (e.g., when $n=3$, the scaling exponent is $2-3=-1$), $\lim_{\lambda \rightarrow0}E(v_\la)=\infty$; this case is supercritical. For the supercritical case, we lose control of the micro behavior of the solution using the energy $E(\cdot)$. While for $n<2$ (e.g., $n=1$), we have $\lim_{\lambda \rightarrow0}E(v_\la)=0$ and this case is subcritical. 

An important criterion for regularity is the Ladyzhenskaya--Prodi--Serrin regularity criterion given in \cite{serrin1962interior}, \cite{prodi1959teorema} and \cite{ladyzhenskaya1967uniqueness}. If a weak solution $v$ satisfies
\[
v \in L^{p}\bigl(0,T; L^{q}(\Omega)\bigr), \quad 
\frac{2}{p} + \frac{3}{q} \leq 1, \quad 2 \le p \le \infty, 3 < q \le \infty,
\]
then $v$ is smooth on $(0,T]$. The endpoint case of $q=3, p=\infty$ can be found in \cite{escauriaza2003}.  


Significant attention has been directed toward \textbf{axisymmetric flows}, for exmaple in \cite{chae2002regularity,chen2009lower1,chen2009lower2,hou2008global,lei2015axially,lei2017criticality,liu2009characterization,zhang2022bounded}, where the velocity and pressure fields are invariant under rotation about a fixed axis. In cylindrical coordinates \((r,\theta,z)\), the velocity takes the form
\[
v = v_r(r,x_3,t)\,e_r + v_\theta(r,x_3,t)\,e_\theta + v_3(r,x_3,t)\,e_3,
\]
where $e_r,e_\theta,e_3$ are the basis in axisymmetric coordinates given by
\be
e_r = (x_1/r, x_2/r, 0), \ e_\theta = ( -x_2/r, x_1/r, 0), \ e_3 = (0,0,1),
\ee
and we have assumed the Cartesian coordinates to be $(x_1,x_2,x_3)$ and $r= \sqrt{x_1^2+x_2^2}$.

Let $D=\{ (r,\theta,x_3): 0\leq r \leq 1, 0 \leq \theta < 2\pi, 0 \leq x_3 \leq 1 \}$ be a cylinder in the cylindrical coordinates. We let the vertical boundary be 
\be
\partial^V D = \{D \cap \{r=1\} \}
\ee
and the horizontal boundary be 
\be
\partial^H D = \{ D \cap \{ x_3 \in \{0,1 \}\} \}.
\ee
We consider the axisymmetric incompressible Navier--Stokes equations with the boundary condition called: the Navier--Hodge--Lions (NHL) boundary condition. Let $\omega = \nabla \times v$ be the vorticity, the NHL boundary condition is given by $v \cdot n = 0$ and $\omega \times n = 0$ ($\omega$ parallel to $n$) on $\partial D$ where $n$ is the unit outward normal vector on $\partial D$. Writing out the axisymmetric
incompressible Navier--Stokes equation and the NHL boundary condition in the coordinate basis we have
\begin{equation} \begin{dcases}
    \label{eq:ASNS}
	&(\Delta - \frac{1}{r^2}) v_r - (v_r \partial_r + v_3 \partial_{x_3} ) v_r + \frac{v_\theta^2}{r} - \partial_r p - \partial_t v_r = 0, \\
	&(\Delta - \frac{1}{r^2}) v_\theta - (v_r \partial_r + v_3 \partial_{x_3} ) v_\theta - \frac{v_\theta v_r}{r} - \partial_t v_\theta = 0, \\
	&\Delta v_3 - (v_r \partial_r + v_3 \partial_{x_3} ) v_3 - \partial_{x_3} p - \partial_t v_3 = 0, \\
    &\frac{1}{r} \partial_r (r v_r) + \partial_{x_3}v_3 =0, \\
	& v_3 = 0, \quad  \omega_r=\omega_\theta =0, \quad \text{on $\partial^H D \times (0,T],$} \\
	&v_r =0, \quad \omega_3=\omega_\theta =0, \quad \text{on $\partial^V D \times (0,T],$} 
\end{dcases} \end{equation}
where $\omega_r,\omega_\theta,\omega_{x_3}$ are the components in the vorticity $\omega$ given by $$\omega = \nabla \times v=(-\partial_{x_3} v_\theta) e_r + (\partial_{x_3}v_r - \partial_r v_3 ) e_\theta + (\frac{v_\theta}{r} + \partial_r v_\theta)e_3 .$$

The derivation of the system \eqref{eq:ASNS} can be found in \cite{zhang2022bounded}. NHL boundary condition was studied in the exterior of a cone \cite{li2024finite}. In \cite{ozanski2025regularity}, the authors obtain a regularity criterion with a slightly different boundary condition.

Historically, it has been known since the work of Ladyzhenskaya, Ukhovskii, and Yudovich that in the \textit{swirl-free} case \(v_\theta \equiv 0\), finite-energy solutions are globally smooth \cite{ukhovskii1968axially,ladyzhenskaya1968unique}. The fundamental challenge lies in the presence of swirl (\(v^\theta \neq 0\)). If the velocity satisfies a Type I bound
\[
|v(x,t)| \leq \frac{C}{r},
\]
then the solution is globally regular \cite{koch2009liouville,chen2009lower1,chen2009lower2}. This condition is scaling-invariant/critical. 

A pivotal conceptual advance was observed by \cite{lei2017criticality} that, under a suitable transformation, the vortex-stretching terms, $\omega \nabla v$, in the axisymmetric Navier--Stokes equations become critical rather than supercritical, narrowing the essential gap between the a priori estimates and the known regularity criteria. 


A key feature of \eqref{eq:ASNS} is the absence of the pressure gradient in the angular equation. Since pressure gradient is non-local which is hard to deal with, this gives an insight that we can start with the angular equation first. The quantity \(\Gamma = r v_\theta\) satisfies
\be
\label{eq:Gamma_og}
\Delta \Gamma - (v_r \partial_r + v_3 \partial_{x_3}) \Gamma - \frac{2}{r}\partial_r \Gamma - \partial_t \Gamma = 0.
\ee
This implies that $\Gamma$ satisfies the maximum principle $\Lnorm{\Gamma(r,x_3,t)}{p} \leq \Lnorm{\Gamma(r,x_3,0)}{p}$. See Lemma \ref{lemma3}.  

We can take the curl of \eqref{eq:ASNS} to obtain the system of equations:
\begin{equation} \begin{dcases}
	&(\Delta - \frac{1}{r^2}) \omega_r - (v_r \partial_r + v_3 \partial_{x_3}) \omega_r + (\omega_r \partial_r + \omega_3 \partial_{x_3} )v_r  - \partial_t \omega_r = 0, \\
	&(\Delta - \frac{1}{r^2}) \omega_\theta - (v_r \partial_r + v_3 \partial_{x_3})  \omega_\theta + \frac{2}{r} v_\theta \partial_{x_3}v_\theta  + \frac{\omega_\theta v_r}{r} - \partial_t \omega_\theta = 0, \\
	&\Delta  \omega_3 - (v_r \partial_r + v_3 \partial_{x_3}) \omega_3 +(\omega_r \partial_r + \omega_3 \partial_{x_3}) v_3  - \partial_t \omega_3= 0. 
    \label{eq:omega}
\end{dcases} \end{equation}

In this paper we will consider two scalar quantities $\Omega = \omega_\theta/r$ and $V = v_\theta/\sqrt{r}$. From \eqref{eq:omega} and \eqref{eq:ASNS} we see that they satisfy the following closed system
\begin{equation} \begin{aligned}
    \label{eq:VOmega}
    &\partial_t \Omega + (v_r \partial_r + v_3 \partial_{x_3}) \Omega = \Delta \Omega + \frac{2}{r} \partial_r \Omega + \frac{\partial_{x_3} (V^2)}{r},  \\
    &\partial_t V + (v_r \partial_r + v_3 \partial_{x_3}) V + \frac{3v_r}{2r}V = \Delta V + \frac{1}{r} \partial_r V - \frac{3}{4r^2}V.
\end{aligned} \end{equation}
From the boundary condition given in \eqref{eq:ASNS}, we see that $\Omega$ and $V$ satisfiy the following boundary conditions
\begin{equation} \begin{aligned}
& \Omega = 0, \, \quad \text{on $\partial D \times (0,T),$} \\
&\partial_{x_3} V = 0, \quad \text{on $\partial^H D \times (0,T),$} \\
&\partial_r V = -\frac{3V}{2r}, \quad \text{on $\partial^V D \times (0,T).$}
\end{aligned} \end{equation}

Apart from the $V$, $\Omega$ system, one can also formulate a closed system in terms of $\Omega$ and $J = \omega_r/r$ for proving global regularity results \cite{lei2017criticality,wei2016regularity}. We follow the method shown by \cite{lei2017criticality} to formulate a closed system in terms of $\Omega$ and $V$ because of its simplicity relative to the formulation in terms of $\Omega$ and $J = \omega_r/r$.

We will also need to use the formulae of gradient of $v$ and the Hessian of $v$. We recall from \cite{zhang2022bounded} if $v$ is an axisymmetric vector then the formula for $\nabla v$ is:
\begin{equation}
\nabla v = 
\begin{pmatrix}
    \partial_r v_r & -\frac{1}{r} v_\theta & \partial_{x_3} v_r \\
\partial_r v_\theta & \frac{1}{r} v_r & \partial_{x_3} v_\theta \\
\partial_r v_3 & 0 & \partial_{x_3} v_3
\end{pmatrix}.
\end{equation}
Replacing $v$ by $\nabla \frac{v_r}{r} = \left(\partial_r \frac{v_r}{r} \right)e_r + \left(\partial_{x_3}\frac{v_r}{r} \right) e_3$, we obtain a formula for the Hessian matrix:
\begin{equation}
\label{eq:Hessian}
\nabla(\nabla \frac{v_r}{r}) = \nabla^2 \frac{v_r}{r} 
= \begin{pmatrix}
\partial_r^2 \frac{v_r}{r} & 0 & \partial_r \partial_{x_3} \frac{v_r}{r} \\
0 & \frac{1}{r} \,\partial_r \frac{v_r}{r} & 0 \\[4pt]
\partial_{x_3}\partial_r \frac{v_r}{r} & 0 & \partial_{x_3}^2 \frac{v_r}{r}
\end{pmatrix}.
\end{equation}

The present work investigates the global regularity of axisymmetric Navier--Stokes flows under the NHL boundary condition in the finite cylinder $D$. The main result is:
\begin{theorem}
    \label{thm:main}
    Consider the system \eqref{eq:ASNS}. Assume the initial value $v_0$ is smooth, divergence free and $v_0$ satisfies the NHL boundary condition. Also, assume
\begin{equation} \begin{aligned}
  \frac{9C_1C_3^{1/2}}{4} \left(\frac12 \| V_0 \|_{L^4}^4 +  \Lnorm{\Omega_0}{2}^{2} \right)^{1/4} \Lnorm{\Gamma_0}{4} \leq \frac{1}{4},
 \label{eq:smallness_condition}
\end{aligned} \end{equation}
where the constants are
\[
C_1 = (1+\sqrt{2}) 2 \sqrt{\pi} (536)^{1/4}(57452(1+5/\pi^2) + 60)^{1/4},
\]
and 
\[
C_3 = \left(  \left(2+\frac{3}{\sqrt{2}} \right)^2+\frac{5}{2} \right)^{1/2}.
\]
Then $v$ is globally regular. 
\end{theorem}

\begin{remark} The condition \eqref{eq:smallness_condition} in Theorem \ref{thm:main} is critical (scaling invariant) under the natural scaling in the Navier--Stokes equation. If $v(x,t)$ is a solution to \eqref{eq:ASNS} then we know that $\tilde{v} = \lambda v(\lambda x, \lambda^2 t)$
 is also a solution. So $v_\theta = \tilde{v}/\lambda$, $r = \lambda \tilde{r}$. We can compute for example 
 \[
 \norm{V_0}_{L^4} = \left( \int_D \frac{v_\theta^4}{r^2} \, dx \right)^{1/4} =  \left( \int_D \frac{1}{\lambda^4} \frac{\tilde{v}_\theta^4}{\lambda^2 \tilde{r}^2} \, \lambda^3 d \tilde{x} \right)^{1/4} = \frac{1}{\lambda^{3/4}} \left( \int_D \frac{\tilde{v}_\theta^4}{
 \tilde{r}^2} \, d\tilde{x} \right)^{1/4} = \frac{1}{\lambda^{3/4}}\norm{\tilde{V}_0}_{L^4}.
\]
Similarly, 
\[
\norm{\Gamma_0}_{L^4} = \lambda^{3/4} \norm{\tilde{\Gamma}_0}_{L^4}, \quad \norm{\Omega_0}_{L^2}^{1/2} = \lambda^{-3/4} \norm{\tilde{\Omega}_0}_{L^2}^{1/2}
\]
\end{remark}

To prove Theorem \ref{thm:main}, we start from the maximum principles for \(\Gamma\). Then by carefully performing the analysis of the boundary behavior, we estimates of the $L^2$ norm of the gradient of $v_r/r$ and the Hessian of $v_r/r$ in terms of the $L^2$ norm of $\Omega$ and $\partial_{x_3}\Omega$ in Lemma \ref{lemma1}. Next, we prove the integral identities in Lemma \ref{lemma4} that enable us to use the Poincar\'e inequalities for $\partial_r (v_r/r)$. Then, in Lemma \ref{lemma:agmon_inequality}, we control the quantity \(\|v_r/r\|_{L^\infty}\) via Agmon-type inequalities. Finally we derive the \(L^4\) bounds for all of the velocity components. The proof culminates in an application of the Ladyzhenskaya--Prodi--Serrin regularity criterion, where we establish that under certain smallness conditions consistent with the critical scaling of ASNS, solutions remain smooth for all time.

The main difficulty in proving Theorem \ref{thm:main} is due to the boundary conditions. In $D$ with the NHL boundary conditions, there will be some extra terms that we need to handle when deriving the energy estimates for $V$ and $\Omega$. Also, most proofs of the existing results do not carry over to the case of $D$ and NHL boundary condition. We need to prove the existing results in a different way, for example in Lemma \ref{lemma1} and \ref{lemma:agmon_inequality}, to comply with the NHL boundary conditions in $D$.

\section{Preliminaries}
In this section, we will introduce some notation and results that will be used to prove the main theorem.

\subsection{Notation}
Here we briefly describe some notation. We will denote $b = v_r e_r + v_3 e_3.$ To simplify the notation, we will use the integral 
\begin{equation}
    \int_D f \, dx := \int_0^1 \int_0^1 f(r,x_3,t) \, r drdx_3.
\end{equation}
$\norm{f}_{L^p(E)}$ denotes the $L^p$ norm for $f$ over a set $E$, i.e. $(\int_E |f|^p \,dx)^{1/p} $. If $E$ is omitted, it means that we are integrating over the set $D$. Sometimes we use $f_0$ or $f_{\theta,0}$ to denote the initial value for $f$ or some component of $f$ if $f$ is a vector.

\subsection{Lemmas}
The following lemma gives a maximum principle for $\Gamma$. In $\mathbb{R}^3$, the proof can be found in Section 3 Proposition 1 in \cite{chae2002regularity}. We give a proof  when the domain is $D$ with the NHL boundary conditions.
\begin{lemma}
\label{lemma3}
If $\Gamma = rv_\theta$ is a smooth function in $D \times [0,T]$ then we have for any  positive even integer $p$ and $t \in [0,T]$,
\begin{equation} 
    \norm{\Gamma(t)}_{L^p} \leq \norm{\Gamma(0)}_{L^p},
\end{equation}
and
\begin{equation} 
    \norm{\Gamma(t)}_{L^\infty} \leq \norm{\Gamma(0)}_{L^\infty}.
\end{equation}
\end{lemma}

\begin{proof}
We first show that for any positive even integer \(p\), the function \(\Gamma^p\) is a subsolution, meaning it satisfies the differential inequality
\begin{equation}
    \label{eq:Gamma_ineq}
    \Delta \Gamma^p - b \cdot \nabla \Gamma^p - \frac{2}{r} \partial_r \Gamma^p - \partial_t \Gamma^p \ge 0 \quad \text{in } D \times [0,T].
\end{equation}

We can compute
\[
\begin{aligned}
\partial_t (\Gamma^p) &= p \Gamma^{p-1} \partial_t \Gamma, \\
\nabla (\Gamma^p) &= p \Gamma^{p-1} \nabla \Gamma, \\
\Delta (\Gamma^p) &= p \Gamma^{p-1} \Delta \Gamma + p(p-1) \Gamma^{p-2} |\nabla \Gamma|^2, \\
\frac{2}{r} \partial_r (\Gamma^p) &= \frac{2}{r} p \Gamma^{p-1} \partial_r \Gamma.
\end{aligned}
\]
Now compute the left‑hand side of \eqref{eq:Gamma_ineq}:
\[
\begin{aligned}
\Delta \Gamma^p - b \cdot \nabla \Gamma^p - \frac{2}{r} \partial_r \Gamma^p - \partial_t \Gamma^p
&= \bigl[ p \Gamma^{p-1} \Delta \Gamma + p(p-1) \Gamma^{p-2} |\nabla \Gamma|^2 \bigr] \\
&\quad - b \cdot \bigl( p \Gamma^{p-1} \nabla \Gamma \bigr) - \frac{2}{r} \bigl( p \Gamma^{p-1} \partial_r \Gamma \bigr) - p \Gamma^{p-1} \partial_t \Gamma \\
&= p \Gamma^{p-1} \bigl( \Delta \Gamma - b \cdot \nabla \Gamma - \frac{2}{r} \partial_r \Gamma - \partial_t \Gamma \bigr) \\
&\quad + p(p-1) \Gamma^{p-2} |\nabla \Gamma|^2.
\end{aligned}
\]
The term in parentheses vanishes because of \eqref{eq:Gamma_og}. Thus
\[
\Delta \Gamma^p - b \cdot \nabla \Gamma^p - \frac{2}{r} \partial_r \Gamma^p - \partial_t \Gamma^p = p(p-1) \Gamma^{p-2} |\nabla \Gamma|^2.
\]
Since \(p\) is a positive even integer, \(p(p-1) \ge 0\) and \(\Gamma^{p-2} \ge 0\). Therefore the right‑hand side is nonnegative, hence for any positive even integer $p$, $\Gamma^p$ is a subsolution:
\begin{equation} 
\Delta \Gamma^p -b \cdot \nabla \Gamma^p -\frac{2}{r}\partial_r\Gamma^p -\partial_t\Gamma^p \geq 0.
\end{equation}
This implies
\begin{equation} \begin{aligned}
    \partial_{t}\int_D \Gamma^{p}(x,t)\,dx &\leq \int_D \Delta \Gamma^{p}(x,t)\,dx - \int_D b \cdot \nabla \Gamma^{p}(x,t)\,dx - \int_D \frac{2}{r}\partial_{r}\Gamma^{p}(x,t)\,dx\\
&\equiv T_{1} + T_{2} + T_{3}. 
\end{aligned} \end{equation}
We consider bounding $T_2$, $T_1$ and $T_3$ respectively.

Using the boundary condition $v_{r} = 0$ on $\partial^V D$, $v_3 = 0$ on $\partial^H D$, also since $\operatorname{div} b = 0$, we see that
\begin{equation}
T_{2} = -\int_D \operatorname{div} (b\,\Gamma^{p}(x,t))\,dx = 0.
\end{equation}

To bound $T_1$, we first use the boundary conditions $\partial_{x_3}\Gamma = r\partial_{x_3}v_{\theta} = 0$ on $\partial^H D$. Therefore
\begin{equation}
T_{1} =\int_D \nabla \cdot ( \nabla \Gamma^p) \, dx = \int_{\partial^{V}D}\partial_{r}\Gamma^{p}\,dx_3 = \int_{\partial^{V}D}p \Gamma^{p-1}\partial_r \Gamma \,dx_3 = 0,
\end{equation}
where we have used the boundary condition on $\partial^V D$, $\partial_r \Gamma = \partial_r(rv_\theta)= r(\frac{v_\theta}{r}+\partial_r v_\theta)=0$.

For $T_3$, since $p$ is even we have
\begin{equation} \begin{aligned}
T_{3} &= -\int_{D}\frac{2}{r}\partial_{r}\Gamma^{p}\,dx \\ &= - 2\int_0^1 (\Gamma^p(1,x_3,t)-\Gamma^p(0,x_3,t)) \,dx_3  \\ & =-2\int_0^1 \Gamma^p(1,x_3,t) \, dx_3 \leq 0.
\end{aligned} \end{equation}

Combining these estimates, we get
\begin{equation}
    \partial_t \norm{\Gamma(t)}_{L^p}^p = p \norm{\Gamma(t)}_{L^p}^{p-1}  \partial_t \norm{\Gamma(t)}_{L^p} \leq 0, \quad \text{$t \in [0,T]$}
\end{equation}
This means that 
\begin{equation}
    \partial_t \norm{\Gamma(t)}_{L^p} \leq 0, \quad \text{$t \in [0,T]$}
\end{equation}
and we obtain $\norm{\Gamma(t)}_{L^p} \leq \norm{\Gamma(0)}_{L^p} $, letting $p \rightarrow\infty$ gives the desired result.
\end{proof}

Lemma \ref{lemma1} gives estimates of the $L^2$ norm of the gradient of $v_r/r$ and the Hessian of $v_r/r$ in terms of the $L^2$ norm of $\Omega$ and $\partial_{x_3}\Omega$. The periodic case can be found in \cite{hou2008global}. It was extended to the global case in \cite{lei2015axially}. The proof for the global case can also be found in \cite{wei2016regularity}. One can also prove it by the formula given in \cite{miao2013global}. It was proven in a special domain with Navier--slip boundary condition in \cite{zhang2022bounded}. Also in the exterior of a cone in \cite{li2024finite}. Due to the domain $D$ and the NHL boundary condition, Lemma \ref{lemma1} does not follow from the mentioned cases, hence we need to present a proof that can avoid the issue caused by the boundary.
\begin{lemma} If $v$ is a smooth solution to \eqref{eq:ASNS}, then
\begin{equation}
    \norm{ \nabla \frac{v_r}{r} }_{L^2} \leq || \Omega ||_{L^2},
    \label{eq:lemma1_estimate1}
\end{equation}
and 
\begin{equation}
    \norm{ \nabla^2 \frac{v_r}{r} }_{L^2} \leq  C_3 || \partial_{x_3} \Omega ||_{L^2},
    \label{eq:lemma1_estimate2}
\end{equation}
where $C_3 = \left(  \left(2+\frac{3}{\sqrt{2}} \right)^2+\frac{5}{2} \right)^{1/2}  $.
\label{lemma1}
\end{lemma}
\begin{proof} The proof is divided into a few steps. In Step 1 we shall prove \eqref{eq:lemma1_estimate1}. In Step 2 we prove 
\begin{equation}
    \norm{ \nabla \partial_{x_3} \frac{v_r}{r} }_{L^2} \leq || \partial_{x_3} \Omega ||_{L^2}.
    \label{eq:lemma1_estimate3}
\end{equation}
Then we use \eqref{eq:lemma1_estimate3} in Step 3 to prove 
\begin{equation}
    \norm{\frac{1}{r} \partial_r\frac{v_r}{r}}_{L^2} \leq \frac{1}{\sqrt{2}} \norm{\partial_{x_3} \Omega}_{L^2}.
\end{equation}
In Step 4 we shall prove
\begin{equation}
    \norm{\partial_r^2 \frac{v_r}{r}}_{L^2} \leq \left(2+\frac{3}{\sqrt{2}} \right)  \norm{\partial_{x_3} \Omega}_{L^2}.
\end{equation}
Finally, in Step 5 we use the definition of Hessian to prove \eqref{eq:lemma1_estimate2}.
  
\textbf{Step 1} From the modified Biot-Savart Law, we know that the following equation is true:
\begin{equation}
    \label{eq:lemma1_Biot}
    \Delta \frac{v_r}{r} + \frac{2}{r}\partial_r \frac{v_r}{r} = \partial_{x_3} \Omega.
\end{equation}
Using $v_r/r$ as a test function on \eqref{eq:lemma1_Biot}, we get
\begin{equation}
    \label{eq:lemma1_1_1}
    \underbrace{\int_D \frac{v_r}{r} \Delta \frac{v_r}{r} \, dx}_{T_1} + 2 \int_D \frac{1}{r} \frac{v_r}{r} \partial_r \frac{v_r}{r} \, dx = \int_D \partial_{x_3} \Omega \frac{v_r}{r} \, dx.
\end{equation}
We can estimate $T_1$ using the boundary conditions $v_r=0$ on $\partial^V D$ and $\partial_{x_3} v_r=0$ on $\partial^H D$. We have
\begin{equation}
    \begin{aligned}
           T_1 = \int_D \frac{v_r}{r} \Delta \frac{v_r}{r} \, dx &= \int_{ D} \operatorname{div}( \frac{v_r}{r}\nabla\frac{v_r}{r} ) \, dx - \int_D \left| \nabla \frac{v_r}{r} \right|^2 \, dx \\
    & =  \int_{\partial^V D} \frac{v_r}{r} \partial_r \frac{v_r}{r} \, dx_3 + \int_{\partial^H D} \frac{v_r}{r} \partial_{x_3} \frac{v_r}{r} \,r dr - \int_D \left| \nabla \frac{v_r}{r} \right|^2 \, dx 
    \\ & = - \int_D \left| \nabla \frac{v_r}{r} \right|^2 \, dx.
    \end{aligned}
\end{equation}
Then we get from \eqref{eq:lemma1_1_1}:
\begin{equation}
\begin{aligned}
    &\int_D \left| \nabla \frac{v_r}{r} \right|^2 \, dx + \int_0^1 \left| \frac{v_r}{r}(0,x_3,t) \right|^2 \, dx_3 
    \\ &= \int_D \Omega \partial_{x_3} \frac{v_r}{r} \, dx 
    \\ &\leq \left( \int_D \Omega^2 \, dx \right)^{1/2} \left( \int_D \left|\partial_{x_3} \frac{v_r}{r} \right|^2 \, dx \right)^{1/2}.
\end{aligned}
\end{equation}
This shows that
\begin{equation}
    \label{eq:lemma1_12}
    \norm{ \nabla \frac{v_r}{r} }_{L^2} \leq \| \Omega \|_2.
\end{equation}

\textbf{Step 2} Using $\partial_{x_3}^2 (v_r/r)$ as a test function on \eqref{eq:lemma1_Biot} we have
\begin{equation}
\begin{aligned}
    \int_D \partial_{x_3} \Omega \, \partial_{x_3}^2 \frac{v_r}{r} \, dx 
    = \underbrace{\int_D \Delta \frac{v_r}{r} \, \partial_{x_3}^2 \frac{v_r}{r} \, dx}_{T_2}
    + \underbrace{\int_D \frac{2}{r} \partial_r \frac{v_r}{r} \partial_{x_3}^2 \frac{v_r}{r} \, dx}_{T_3}.
\end{aligned}
\end{equation}
We estimate $T_3$. We note that $\partial_{x_3}v_r =0$ on $\partial D$ and since $v_r =0$ on $\partial^V D$, we have $\partial_{x_3} v_r = 0$ on $\partial^V D$:
\begin{equation}
\begin{aligned}
    T_3 &= 2 \int_0^1 \int_0^1 \partial_{x_3}( \partial_r \frac{v_r}{r} \partial_{x_3} \frac{v_r}{r}) \, drdx_3 - 2 \int_0^1 \int_0^1 \partial_r \partial_{x_3} \frac{v_r}{r} \partial_{x_3} \frac{v_r}{r} \, drdx_3
    \\ & = - 2 \int_0^1 \int_0^1 \partial_r \partial_{x_3} \frac{v_r}{r} \partial_{x_3} \frac{v_r}{r} \, drdx_3 
    \\ & = - \int_0^1 \int_0^1 \partial_r ( \partial_{x_3}\frac{v_r}{r})^2 \, drdx_3
    \\ & = - \int_0^1 (\partial_{x_3} \frac{v_r}{r})^2(1,x_3,t)\, dx_3 + \int_0^1 (\partial_{x_3} \frac{v_r}{r})^2(0,x_3,t)\, dx_3
    \\ & = \int_0^1 (\partial_{x_3} \frac{v_r}{r})^2(0,x_3,t)\, dx_3.
\end{aligned}
\end{equation}

Next we estimate $T_2$:
\begin{equation}
\begin{aligned}
    T_2 &= \int_D \Delta \frac{v_r}{r} \, \partial_{x_3}^2 \frac{v_r}{r} \, dx 
    \\ &= \int_{\partial^V D} \partial_r\frac{v_r}{r} \partial_{x_3}^2 \frac{v_r}{r} \, rdx_3  + \int_{\partial^H D} \partial_{x_3}\frac{v_r}{r} \partial_{x_3}^2 \frac{v_r}{r} \, rdr - \int_D \nabla\frac{v_r}{r} \partial_{x_3}\nabla \partial_{x_3} \frac{v_r}{r} \, dx 
    \\ & = - \int_D \nabla\frac{v_r}{r} \partial_{x_3}\nabla \partial_{x_3} \frac{v_r}{r} \, dx 
    \\ &= - \int_D \partial_{x_3}( \nabla \frac{v_r}{r} \nabla\partial_{x_3}\frac{v_r}{r}) \, dx + \int_D |\nabla \partial_{x_3} \frac{v_r}{r}|^2 \, dx
    \\ & = - \int_0^1 \left. \nabla \frac{v_r}{r}\nabla \partial_{x_3}\frac{v_r}{r} \right|_{x_3 = 0}^{ x_3 = 1} rdr + \int_D |\nabla \partial_{x_3} \frac{v_r}{r}|^2 \, dx
    \\ & = - \int_0^1 \left. (\partial_r \frac{v_r}{r} e_r + 0 e_3) \cdot (0e_r + \partial_{x_3}^2 \frac{v_r}{r}e_3) \right|_{x_3 = 0}^{ x_3 = 1} rdr + \int_D |\nabla \partial_{x_3} \frac{v_r}{r}|^2 \, dx
    \\ & = \int_D |\nabla \partial_{x_3} \frac{v_r}{r}|^2 \, dx,
\end{aligned}
\end{equation}
where in the second equality the boundary terms vanish because $\partial_{x_3} \frac{v_r}{r} = 0$ on $\partial D$ and $\partial_{x_3}^2 \frac{v_r}{r} = 0$ on $\partial^V D$. In the sixth equality, we have used $\partial_{x_3} \frac{v_r}{r}=0$ on $\partial^H D$ and $\partial_r \partial_{x_3} \frac{v_r}{r} = 0$ on $\partial^H D$.

The combination of the estimation of $T_2$ and $T_3$ gives
\begin{equation}
\begin{aligned}
    &\int_D | \nabla \partial_{x_3} \frac{v_r}{r} |^2 \, dx + \int_0^1 (\partial_{x_3} \frac{v_r}{r})^2(0,x_3,t)\, dx_3 
    \\ &= \int_D \partial_{x_3}\Omega \, \partial_{x_3}^2 \frac{v_r}{r} \, dx 
    \\ &\leq  \left( \int_D \left|\partial_{x_3} \Omega \right|^2 \, dx \right)^{1/2} \left( \int_D \left|\nabla \partial_{x_3} \frac{v_r}{r} \right|^2 \, dx\right)^{1/2}.
\end{aligned}
\end{equation}
This shows that
\begin{equation}
    \label{eq:lemma1_11}
    \norm{ \nabla \partial_{x_3} \frac{v_r}{r} }_{L^2} \leq || \partial_{x_3} \Omega ||_{L^2}.
\end{equation}

\textbf{Step 3} We want to show that 
\begin{equation}
    \label{eq:lemma1_7}
    \norm{\frac{1}{r} \partial_r\frac{v_r}{r}}_{L^2}^2 \leq \frac{1}{2} \norm{\partial_{x_3} \Omega}_{L^2}^2.
\end{equation}
Again using the equation
\begin{equation}
    \Delta \frac{v_r}{r} + \frac{2}{r}\partial_r \frac{v_r}{r} = \partial_{x_3} \Omega,
\end{equation}
we have 
\begin{equation}
    \label{eq:lemma1_5}
    \partial_r^2\frac{v_r}{r} + \frac{3}{r}\partial_r \frac{v_r}{r} + \partial_{x_3}^2 \frac{v_r}{r}= \partial_{x_3} \Omega,
\end{equation}
or equivalently
\begin{equation}
    \label{eq:lemma1_8}
    \frac{1}{r^3} \partial_r \left( r^3 \partial_r \frac{v_r}{r} \right) + \partial_{x_3}^2 \frac{v_r}{r} = \partial_{x_3}\Omega.
\end{equation}

Recall the regularity result for any smooth axisymmetric vector on the axis (\cite{liu2009characterization}, Corollary 1), all the even-order derivatives with respect to $r$ of $v_r$ at $r=0$ vanish. Hence, the expansion of $v_r$ around $r=0$ is 
\begin{equation}
    c_1r + c_2 r^3 + c_3 r^5 + \cdots,
\end{equation}
where $c_i=c_i(x_3,t)$. Then the expansion of $v_r/r$ around $r=0$ is 
\begin{equation}
    c_1 + c_2 r^2 + c_3 r^4 + \cdots,
\end{equation}
and thus
\begin{equation}
    \partial_r\frac{v_r}{r} = 0,\quad{\text{at $r=0$.}}
\end{equation}

Rearranging \eqref{eq:lemma1_8} and integrating from $0$ to $r$, using $\partial_r (v_r/r)(0,x_3,t) = 0$, we have
\begin{equation}
    \left( \frac{1}{r} \partial_r \frac{v_r}{r} \right) (r,x_3,t)  = \frac{1}{r^4} \int_0^r s^3 \underbrace{ \left(\partial_{x_3}\Omega(s,x_3,t) - \partial_{x_3}^2 \frac{v_r(s,x_3,t)}{s} \right)  }_{=:g(s)} \, ds,
\end{equation}
where we have denoted $$g(s,x_3,t) := \partial_{x_3}\Omega(s,x_3,t) - \partial_{x_3}^2 \frac{v_r(s,x_3,t)}{s}. $$ We have
\begin{equation}
    \frac{1}{r} \partial_r \frac{v_r}{r} = 
    \frac{1}{r^4} \int_0^r s^3 g(s,x_3,t)\,ds.
\end{equation}
Using H\"{o}lder's inequality we have
\begin{equation}
\begin{aligned}
    \left| \frac{1}{r} \partial_r \frac{v_r}{r} \right|^2=\left| \frac{1}{r^4} \int_0^r s^3 g(s,x_3) \, ds \right|^2
     & =\frac{1}{r^8}  \left| \int_0^r s^3 g(s,x_3) \, ds \right|^2
     \\ & \leq \frac{1}{r^8} \left( \int_0^r s^3 \, ds \right) \left( \int_0^r s^3 g^2(s,x_3) \,ds \right) 
      \\ &  = \frac{1}{4r^4} \int_0^r s^3 g^2(s,x_3) \, ds.
\end{aligned}
\end{equation}
Using this and integrating $|\frac{1}{r}\partial_r \frac{v_r}{r}|^2r$ with respect to $r$ gives
\begin{equation}
    \int_0^1 \left|\frac{1}{r}\partial_r \frac{v_r}{r} \right|^2 \,rdr \leq \frac{1}{4}  \int_0^1 \frac{1}{r^4} \int_0^r s^3 g^2(s,x_3) \, ds \, rdr = \frac{1}{4}  \int_0^1 \int_0^r s^3 g^2(s,x_3) \frac{1}{r^3}  \, ds dr.
\end{equation}
We switch the order of integration for the last integral which yields
\begin{equation}
\begin{aligned}
    \frac{1}{4}  \int_0^1 \int_0^r s^3 g^2(s,x_3,t) \frac{1}{r^3}  \, ds dr 
    & = \frac{1}{4} \int_0^1 \int_s^1 s^3 g^2(s,x_3,t) \frac{1}{r^3}  \, dr ds 
    \\ & = \frac{1}{4} \int_0^1  s^3 g^2(s,x_3,t) \int_s^1 \frac{1}{r^3}  \, dr ds 
    \\ & = \frac{1}{4} \int_0^1  s^3g^2(s,x_3,t) \left(\frac{1}{2s^2} - \frac{1}{2} \right) \, ds
    \\& = \frac{1}{8} \int_0^1 (s-s^3) g^2(s,x_3,t)\,ds 
    \\ & \leq \frac{1}{8} \int_0^1 s g^2(s,x_3,t)\,ds .
\end{aligned}
\end{equation}
Replacing $s$ with $r$, we get
\begin{equation}
\begin{aligned}
    \int_0^1 \left| \frac{1}{r} \partial_r \frac{v_r}{r} \right|^2 \, rdr \leq \frac{1}{8} \int_0^1  g^2(r,x_3,t) \, rdr
\end{aligned}
\end{equation}
Integrating in $x_3$ we get
\begin{equation}
    \norm{\frac{1}{r}\partial_r \frac{v_r}{r}}_{L^2}^2 \leq \frac{1}{8} \norm{g}_{L^2}^2.
\end{equation}
Recalling that $g = \partial_{x_3}\Omega(s,x_3,t) - \partial_{x_3}^2 \frac{v_r(s,x_3,t)}{s}$ and using \eqref{eq:lemma1_11}, we obtain
\begin{equation}
\begin{aligned}
    \norm{\frac{1}{r}\partial_r \frac{v_r}{r}}_{L^2} ^2
    & \leq \frac{1}{8} \norm{g}^2_{L^2}
    \\ & \le \frac{1}{4} \left( \norm{\partial_{x_3} \Omega}^2_{L^2} +\norm{\partial_{x_3}^2 \frac{v_r}{r}}^2_{L^2} \right)
    \\ & \leq \frac{1}{2} \norm{\partial_{x_3} \Omega}^2_{L^2}.
\end{aligned}
\end{equation}

\textbf{Step 4} We want to show that 
\begin{equation}
    \label{eq:lemma1_6}
    \norm{\partial_r^2 \frac{v_r}{r}}_{L^2} \leq \left(2+\frac{3}{\sqrt{2}} \right)  \norm{\partial_{x_3} \Omega}_{L^2}.
\end{equation}
From \eqref{eq:lemma1_5}, we have
\begin{equation}
        \partial_r^2\frac{v_r}{r} = \partial_{x_3} \Omega - \frac{3}{r}\partial_r \frac{v_r}{r} - \partial_{x_3}^2 \frac{v_r}{r}.
\end{equation}
Then
\begin{equation}
\begin{aligned}
    \norm{\partial_r^2\frac{v_r}{r}}_{L^2} 
    & = \norm{\partial_{x_3} \Omega - \frac{3}{r}\partial_r \frac{v_r}{r} - \partial_{x_3}^2 \frac{v_r}{r}}_{L^2} 
    \\ & \leq \norm{\partial_{x_3} \Omega}_{L^2} + 3\norm{\frac{1}{r}\partial_r \frac{v_r}{r} }_{L^2} + \norm{\partial_{x_3}^2 \frac{v_r}{r}}_{L^2} 
    \\ & \leq \norm{\partial_{x_3} \Omega}_{L^2} + \frac{3}{\sqrt{2}} \norm{\partial_{x_3} \Omega}_{L^2} + \norm{\partial_{x_3} \Omega}_{L^2}
    \\ & = \left(2+\frac{3}{\sqrt{2}} \right) \norm{\partial_{x_3} \Omega}_{L^2}.
\end{aligned}
\end{equation}

\textbf{Step 5} From \eqref{eq:Hessian}, we see that
\begin{equation}
    \label{eq:lemma1_4}
    \norm{\nabla^2 \frac{v_r}{r}}_{L^2}^2 = \norm{\partial_r^2 \frac{v_r}{r}}_{L^2}^2 + \norm{ \frac{1}{r} \partial_r \frac{v_r}{r}}_{L^2}^2 + \norm{\partial_{x_3}^2 \frac{v_r}{r}}_{L^2}^2 + 2 \norm{\partial_r \partial_{x_3} \frac{v_r}{r}}_{L^2}^2.
\end{equation}
Using \eqref{eq:lemma1_estimate1}, \eqref{eq:lemma1_11}, \eqref{eq:lemma1_6} and \eqref{eq:lemma1_7} we have
\begin{equation}
\begin{aligned}
    \norm{\nabla^2 \frac{v_r}{r}}_{L^2}^2 
    & = \norm{\partial_r^2 \frac{v_r}{r}}_{L^2}^2 + \norm{ \frac{1}{r} \partial_r \frac{v_r}{r}}_{L^2}^2 + \norm{\partial_{x_3}^2 \frac{v_r}{r}}_{L^2}^2 + 2 \norm{\partial_r \partial_{x_3} \frac{v_r}{r}}_{L^2}^2 
    \\ & \leq \left(2+\frac{3}{\sqrt{2}} \right)^2 \norm{\partial_{x_3} \Omega}_{L^2}^2 + \frac{1}{2} \norm{\partial_{x_3} \Omega}_{L^2}^2  + 2 \norm{\partial_{x_3} \Omega}_{L^2}^2 
    \\& \leq \left( \left(2+\frac{3}{\sqrt{2}} \right)^2+\frac{5}{2} \right) \norm{\partial_{x_3} \Omega}_{L^2}^2.
\end{aligned}
\end{equation}
Hence \eqref{eq:lemma1_estimate2} is true.
\end{proof}

We observe that $v_r/r$ is the rate of change per unit area in the radial direction. Lemma \ref{lemma4} describes that the rate of change $v_r/r$ in the radial direction is in some sense balanced in the whole domain. Imagine a elastic ring in $D$, affected by the fluid, that is expanding fast in some region, then the ring must be expanding slowly, or even contracting in some other region. This proof is a slight modification from the same result in \cite{zhang2022bounded}.
\begin{lemma} If $v$ is a smooth solution to \eqref{eq:ASNS}, then for any $r,t$
    \begin{equation}
    \int_0^1 \left( \frac{v_r}{r} \right)(r,x_3,t) \, dx_3 = 0.
    \end{equation}
\label{lemma4}
\end{lemma}
\begin{proof}
We start the proof from the modified Biot-Savart Law. We know that 
\begin{equation}
    (\Delta - \frac{1}{r^2}) v_r = \partial_{x_3} \omega_\theta
\end{equation}
Let $h = \frac{v_r}{r}$, then we have 
\begin{equation}
    \label{eq:lemma_agmon_2}
    \partial_r^2 h + \frac{3}{r}\partial_r h + \partial_{x_3}^2 h  = \frac{1}{r}\partial_{x_3} \omega_\theta.
\end{equation}
By the boundary conditions $\omega_\theta=0$ and $\partial_{x_3}v_r = \partial_r v_3 =0 $ on $\partial^H D$.
\begin{equation}
    \int_0^1 \partial_{x_3} \omega_\theta \, dx_3 = \omega_\theta(r,1,t) - \omega_\theta(r,0,t) = 0,
\end{equation}
and
\begin{equation}
    \int_0^1 \partial_{x_3}^2 h\, dx_3 = \frac{1}{r}\int_0^1 \partial_{x_3}^2 v_r \, dx_3 = \frac{1}{r} \left( \partial_{x_3} v_r(r,1,t) - \partial_{x_3} v_r(r,0,t) \right) = 0.
\end{equation}
Integrating \eqref{eq:lemma_agmon_2} gives
\begin{equation}
    \partial_r^2  \int_0^1 h \, dx_3 + \frac{3}{r}\partial_r \int_0^1 h \, dx_3 = 0.
\end{equation}
Let $g(r,t) = \int_0^1 h \, dx_3$. We have an ODE
\begin{equation}
    \partial_r^2 g(r,t) + \frac{3}{r} \partial_r g(r,t) = 0.
\end{equation}
We look for solution in the form of $g(r,t) = c(t)r^m$ for some constant $m$. If we assume $g$ in this form, we found that the general solution to the ODE is 
\begin{equation}
    g(r,t) = c_1(t) + c_2(t) r^{-2}.
\end{equation}
We note that $c_2(t)=0$ for all $t$ since otherwise the solution is not smooth at $r=0$. From the boundary condition at $r=1$, $v_r = 0$ on $\partial^V D$, hence for any $t$,
\begin{equation}
    g(1,t) = \int_0^1 v_r(1,x_3,t) \, dx_3 = c_1(t) = 0.
\end{equation}
Thus we have for any $r,t$,
\begin{equation}
     \int_0^1 \frac{v_r}{r} \, dx_3 = g(r,t) = 0.
\end{equation}
\end{proof}

Since for any $r,t$,
\begin{equation}
    \int_0^1 \frac{v_r}{r} dx_3 = 0
\end{equation}
from this we infer
\begin{equation}
     \int_0^1 \partial_r \frac{v_r}{r} dx_3 = \partial_r \int_0^1 \frac{v_r}{r} dx_3 = 0.
\end{equation}
and thus
\begin{equation}
     \int_0^1 r \int_0^1 \partial_r  \frac{v_r}{r} \, dx_3dr = 0.
\end{equation}
This prove that $\int_D \partial_r \frac{v_r}{r} \, dx =0.$ So we can apply the Poincar\'e inequality for $\partial_r \frac{v_r}{r}$ in $D$.

Next, we present a proof of the Agmon's inequality in $\mathbb{R}^3$ \cite{agmon2010lectures}. Although the proof of this inequality can be found in the literature, we want to compute an explicit constant for the inequality.
\begin{lemma}
    \label{lemma:agmonR^3}
    If $f \in H^2(\mathbb{R}^3)$, then 
    \begin{equation}
        \norm{f}_{L^\infty(\mathbb{R}^3)} \leq (1+\sqrt{2}) 2 \sqrt{\pi}\norm{f}^{1/2}_{H^1(\mathbb{R}^3)} \norm{f}^{1/2}_{H^2(\mathbb{R}^3)}
    \end{equation}
\end{lemma}
\begin{proof}
    Set $M = \norm{f}_{H^1(\mathbb{R}^3)}^{-1} \norm{f}_{H^2(\mathbb{R}^3)}$. Let $\hat{f}(\xi)$ be the Fourier transform of $f$. Using Cauchy-Schwarz inequality on the Fourier expansion of $f$ gives
    \begin{equation}
    \begin{aligned}
    \label{eq:lemma_agmonR^3_1}
                f(x) &= \int \hat{f}(\xi) e^{2\pi i\xi \cdot x} \, d\xi 
                \\ & = \int_{|\xi| \leq M} \hat{f}(\xi) e^{2\pi i\xi \cdot x} \, d\xi + \int_{|\xi| \geq M} \hat{f}(\xi) e^{2\pi i\xi \cdot x} \, d\xi 
                \\ & = \int_{|\xi| \leq M} \hat{f}(\xi) e^{2\pi i\xi \cdot x} (1+|\xi|)(1+|\xi|)^{-1} \, d\xi + \int_{|\xi| \geq M} \hat{f}(\xi) e^{2\pi i\xi \cdot x} |\xi|^2 |\xi|^{-2} \, d\xi 
                \\& \leq \left( \int_{|\xi| \leq M} |\hat{f}|^2(\xi) (1+|\xi|)^2\, d\xi \right)^{1/2} \left( \int_{|\xi| \leq M} (1+|\xi|)^{-2} \, d\xi \right)^{1/2} 
                \\ &+ \left( \int_{|\xi| \geq M} |\hat{f}|^2(\xi) |\xi|^4\, d\xi \right)^{1/2} \left( \int_{|\xi| \geq M}  |\xi|^{-4} \, d\xi \right)^{1/2}
                \\& \leq \sqrt{2}\norm{f}_{H^1(\mathbb{R}^3)} \left( \int_{|\xi| \leq M} (1+|\xi|)^{-2} \, d\xi \right)^{1/2} + \norm{f}_{H^2(\mathbb{R}^3)} \left( \int_{|\xi| \geq M}  |\xi|^{-4} \, d\xi \right)^{1/2}.
    \end{aligned}
    \end{equation}

Next we can compute in spherical coordinates $(\rho,\phi_1,\phi_2)$ that
\begin{equation}
\begin{aligned}
    &\int_{|\xi| \leq M} (1+|\xi|)^{-2} \, d\xi  
    \\ &= \int_0^M \int_0^\pi \sin\phi_2 \, d\phi_2 \int_0^{2\pi } \, d\phi_1 \frac{\rho^2}{(1+\rho)^2} \, d\rho
    \\ &= 4\pi \int_0^M \frac{\rho^2}{(1+\rho)^2} \, d\rho  
    \\ & \leq 4\pi \int_0^M \frac{\rho^2}{1+\rho^2} \, d\rho 
    \\ &= 4 \pi( M - \tan^{-1}(M)) \leq 4\pi M, 
\end{aligned}
\end{equation}
and
\begin{equation}
\begin{aligned}
    \int_{|\xi| \geq M} |\xi|^{-4} \, d\xi = 4\pi \int_M^\infty \rho^{-2}\, d\rho = 4\pi M^{-1}.
\end{aligned}
\end{equation}
Putting this into \eqref{eq:lemma_agmonR^3_1} we obtain
\begin{equation}
\begin{aligned}
    |f(x)| \leq  2\sqrt{2\pi} M^{1/2} \norm{f}_{H^1(\mathbb{R}^3)}  + 2\sqrt{\pi} M^{-1/2} \norm{f}_{H^2(\mathbb{R}^3)}.
\end{aligned}
\end{equation}
Since $M = \norm{f}_{H^1(\mathbb{R}^3)}^{-1} \norm{f}_{H^2(\mathbb{R}^3)}$, we have
\begin{equation} \begin{aligned}
    |f(x)| &\leq  2\sqrt{2\pi} \norm{f}_{H^1(\mathbb{R}^3)}^{1/2} \norm{f}_{H^2(\mathbb{R}^3)}^{1/2} +2\sqrt{\pi} \norm{f}_{H^1(\mathbb{R}^3)}^{1/2} \norm{f}_{H^2(\mathbb{R}^3)}^{1/2} 
    \\ &= (1+\sqrt{2}) 2 \sqrt{\pi} \norm{f}_{H^1(\mathbb{R}^3)}^{1/2} \norm{f}_{H^2(\mathbb{R}^3)}^{1/2}.
\end{aligned} \end{equation}
This implies
\begin{equation}
    \norm{f}_{L^\infty(\mathbb{R}^3)} \leq (1+\sqrt{2}) 2 \sqrt{\pi} \norm{f}_{H^1(\mathbb{R}^3)}^{1/2} \norm{f}_{H^2(\mathbb{R}^3)}^{1/2}.
\end{equation}
\end{proof}

Define the cutoff function in the $x_3$-direction by
\be
\eta_1(x_3) =
\begin{dcases}
    0, \quad x_3 \leq -1/2, \\
    -16x_3^3-12x_3^2+1, \quad -1/2 \leq x_3  \leq 0, \\
    1, \quad 0 \leq x_3 \leq 1, \\
    16x_3^3 -60x_3^2 + 72x_3 -27, \quad 1 \leq x_3 \leq 3/2, \\
    0, \quad x_3 \geq 3/2.
\end{dcases}
\ee
We have $\eta_1 \leq 1$ and we can check by direct computation that 
\be
|\partial_{x_3}\eta_1(x_3)| \leq 3, \quad |\partial^2_{x_3}\eta_1(x_3)| \leq 24.
\ee
Also, define the cutoff function in the $r$-direction by
\be
\eta_2(r) =
\begin{dcases}
    1, \quad 0 \leq r \leq 1, \\
    16r^3 -60r^2 + 72r -27, \quad 1 \leq r \leq 3/2, \\
    0, \quad r \geq 3/2.
\end{dcases}
\ee
$\eta_2 \leq 1$ and it satisfies the same bound as $\eta_1$. Namely $|\partial_r\eta_2| \leq 3$ and $|\partial^2_r\eta_2| \leq 24$. We will use $\eta_1(x_3)$ and $\eta_2(r)$ in the proof of Lemma \ref{lemma:agmon_inequality}.

The $\mathbb{R}^3$ case of Lemma \ref{lemma:agmon_inequality} can be found in  \cite{lei2017criticality}. Due to the domain $D$ and the boundary conditions, we need to provide another proof.
\begin{lemma} If $v$ is a smooth solution to \eqref{eq:ASNS} then
\label{lemma:agmon_inequality}
\begin{equation}
    \Lnorm{\frac{v_r}{r}}{\infty(D)} \leq C_1 \Lnorm{\nabla \frac{v_r}{r}}{2(D)}^{1/2} \Lnorm{\nabla^2 \frac{v_r}{r}}{2(D)}^{1/2}
\end{equation}
where $C_1 = (1+\sqrt{2}) 2 \sqrt{\pi} (536)^{1/4}(57452(1+5/\pi^2) + 60)^{1/4} $.
\end{lemma}

\begin{proof}
We will first do an even extension in the $x_3$-direction and then do an odd extension in the $r$-direction for $f=v_r/r$. Then we prove some Poincar\'e inequalities. Finally we combine the extended function and the Poincar\'e inequalities with Lemma \ref{lemma:agmonR^3} to prove this lemma.

Let $f= v_r/r$ and let $\tilde{D}_1 = \{(r,\theta,x_3):0 \leq r \leq 1, 0 \leq \theta < 2\pi, -1/2 \leq x_3 \leq 3/2 \}$. Define the function on $\tilde{D}_1$ by
\be 
\tilde{f}(r,x_3,t) = 
\begin{dcases}
    f(r,-x_3,t), \quad  -1/2 \leq x_3 \leq 0\\
    f(r,x_3,t), \quad 0 \leq x_3 \leq 1\\
    f(r,2-x_3,t), \quad  1 \leq x_3 \leq 3/2. \\
\end{dcases}
\ee
We consider $F = \tilde{f}(r,x_3,t)\eta_1(x_3)$ defined on $\tilde{D}_1$. Then
\be \ba
&\| F \|^2_{L^2(\tilde{D}_1)} \leq 2 \| f \|^2_{L^2(D)}, \\
&\| \nabla F \|^2_{L^2(\tilde{D}_1)} \leq 6 \| \nabla f \|^2_{L^2(D)} + 6(\sup_{x_3}|\partial_{x_3}\eta_1|)^2  \| f \|^2_{L^2(D)} \leq 6 \| \nabla f \|^2_{L^2(D)} + 54 \| f \|^2_{L^2(D)}, 
\ea \ee
And by using formula \eqref{eq:Hessian} we have
\be \ba
\| \nabla^2 F \|^2_{L^2(\tilde{D}_1)} 
\leq & 2 \| \partial_r^2 f \|_{L^2(D)}^2 + 8 \| \partial_r\partial_{x_3} f \|_{L^2(D)}^2 + 2 \| \frac{1}{r} \partial_r f \|_{L^2(D)}^2 + 6 \| \partial_{x_3}^2 f \|_{L^2(D)}^2  
\\ & +  8 (\sup_{x_3}|\partial_{x_3}\eta_1|)^2 \| \partial_r f \|_{L^2(D)}^2  + 24 (\sup_{x_3}|\partial_{x_3}\eta_1|)^2 \| \partial_{x_3} f \|_{L^2(D)}^2 \\ & + 6 (\sup_{x_3}|\partial_{x_3}^2\eta_1|)^2 \| f \|_{L^2(D)}^2.
\\& \leq 3456\| f \|^2_{L^2(D)}  + 216\| \nabla f \|^2_{L^2(D)} + 6 \| \nabla^2 f \|^2_{L^2(D)}.
\ea \ee
Next, we further extend $F$ in the $r$-direction. Let $\tilde{D}_2 = \{(r,\theta,x_3):0 \leq r \leq 3/2, 0 \leq \theta < 2\pi, -1/2 \leq x_3 \leq 3/2 \} $. Define $G$ on $\tilde{D}_2$ by
\be
G(r,x_3,t)=
\begin{cases}
    F(r,x_3,t), \quad 0\leq r\leq 1 \\
    -F(2-r,x_3,t)\eta_2(r), \quad 1 \leq r \leq 3/2 \\
    0, \quad r \geq 3/2.
\end{cases}
\ee
The integral of $G$ is still $\int G \, rdrdx_3$. Because of the weight $r$ in the cylindrical integral, we have
\be
\ba
\| G \|_{L^2(\tilde{D}_2)}^2 = \| G \|_{L^2(\tilde{D}_1)}^2 + \| G \|_{L^2(1\leq r \leq 3/2)}^2 \leq \| F \|_{L^2(\tilde{D}_1)}^2 + 3 \| F \|_{L^2(\tilde{D}_1)}^2 = 4\| F \|_{L^2(\tilde{D}_1)}^2.
\ea
\ee
Similarly we have
\be
\| \nabla  G \|_{L^2(\tilde{D}_2)}^2 \leq 7 \| \nabla F \|_{L^2(\tilde{D}_1)}^2 + 6 (\sup_r |\partial_r \eta_2|)^2 \| F \|_{L^2(\tilde{D}_1)}^2.
\ee
And again by using formula \eqref{eq:Hessian} we have
\be
\ba
& \| \nabla^2 G \|_{L^2(\tilde{D}_2)}^2 \\
& \leq \| \nabla^2 F \|_{L^2(\tilde{D}_1)}^2 + 3 \| \nabla^2 F(2-r,\cdot,\cdot) \|_{L^2(1 \leq r \leq 3/2)}^2 + 6 (\sup_r |\partial_r \eta_2|)^2 \| \nabla F(2-r,\cdot,\cdot) \|_{L^2(1 \leq r \leq 3/2)}^2
 \\ & + (2 (\sup_r |\partial_r \eta_2|)^2 + 3 (\sup_r  |\partial_r^2 \eta_2|)^2 ) \|  F(2-r,\cdot,\cdot) \|_{L^2(1 \leq r \leq 3/2)}^2 \\
 & \leq 10 \| \nabla^2 F \|_{L^2(\tilde{D}_1)}^2 + 18 (\sup_r |\partial_r \eta_2|)^2 \| \nabla F \|_{L^2(\tilde{D}_1)}^2 + 3(2 (\sup_r |\partial_r \eta_2|)^2 + 3 (\sup_r  |\partial_r^2 \eta_2|)^2 ) \|  F \|_{L^2(\tilde{D}_1)}^2
\ea
\ee

On the other hand, since $v_r(1,x_3,t)=0$, we have
\begin{equation} \begin{aligned}
    &\int_0^1 \int_0^1 \left|\frac{v_r}{r} \right|^2 \, rdrdx_3 
    \\ &= \int_0^1 \int_0^1 \left|\frac{v_r}{r}\right|^2 \frac{d}{dr}\frac{r^2}{2} \, drdx_3 = - \int_0^1 \int_0^1 \frac{v_r}{r} \left( \partial_r \frac{v_r}{r} \right)  r^2 \,drdx_3 
    \\ &\leq  \left(\int_0^1\int_0^1  \left|\partial_r \frac{v_r}{r}\right|^2 \, rdrdx_3 \right)^{1/2} \left( \int_0^1\int_0^1 \left|\frac{v_r}{r}\right|^2 \, rdrdx_3 \right)^{1/2}.
\end{aligned} \end{equation}
That is 
\begin{equation} \begin{aligned}
 \Lnorm{\frac{v_r}{r}}{2} \leq  \Lnorm{\nabla \frac{v_r}{r}}{2}.
\end{aligned} \end{equation}
The same proof will also work for $\partial_{x_3}\frac{v_r}{r}$ since by the boundary condition $v_r(1,x_3,t)=0$ hence $\partial_{x_3}\frac{v_r}{r}(1,x_3,t)=0$. Then
\begin{equation} \begin{aligned}
    &\int_0^1 \int_0^1 \left| \partial_{x_3}\frac{v_r}{r} \right|^2 \, rdrdx_3 
    \\ &= \int_0^1 \int_0^1 \left| \partial_{x_3}\frac{v_r}{r} \right|^2 \frac{d}{dr}\frac{r^2}{2} \, drdx_3 = - \int_0^1\int_0^1 \partial_{x_3}\frac{v_r}{r} \left( \partial_r \partial_{x_3}\frac{v_r}{r} \right)  r^2 \,drdx_3 
    \\ &\leq  \left(\int_0^1\int_0^1 \left|\partial_r \partial_{x_3}\frac{v_r}{r} \right|^2 \, rdrdx_3 \right)^{1/2} \left( \int_0^1  \int_0^1 \left|\partial_{x_3}\frac{v_r}{r} \right|^2 \, rdrdx_3 \right)^{1/2}.
\end{aligned} \end{equation}
From this we obtain
\begin{equation} \begin{aligned}
  \Lnorm{\partial_{x_3}\frac{v_r}{r}}{2}\leq  \Lnorm{\nabla^2 \frac{v_r}{r}}{2}.
  \label{eq:lemma_agmon_1}
\end{aligned} \end{equation}

Since we know from Lemma \ref{lemma4} that $\int_D \partial_r \frac{v_r}{r} \, dx =0$, we can apply the Poincar\'e inequality and get
\begin{equation}
    \norm{\partial_r \frac{v_r}{r}}_{L^2} \leq C_{P}\norm{\nabla \partial_r \frac{v_r}{r} }_{L^2} \leq C_{P}\norm{\nabla^2 \frac{v_r}{r} }_{L^2}
    \label{eq:lemma_agmon_3}
\end{equation}
where $C_P$ is the constant in the Poincar\'e inequality.

Combining \eqref{eq:lemma_agmon_1} and \eqref{eq:lemma_agmon_3} we conclude that
\begin{equation}
    \norm{\nabla \frac{v_r}{r}}_{L^2}^2 = \norm{\partial_r \frac{v_r}{r}}_{L^2}^2 + \norm{\partial_{x_3} \frac{v_r}{r}}_{L^2}^2 \leq (1+C_{P}^2) \norm{\nabla^2 \frac{v_r}{r} }_{L^2}^2 \leq (1 + 5/\pi^2)\norm{\nabla^2 \frac{v_r}{r} }_{L^2}^2,
\end{equation}
where we have used the estimate for $C_p$ in $D$: $C_p \le \sqrt{5}/\pi$ \cite{payne1960optimal}.

By construction of $G$, we can extend $G$ to $\mathbb{R}^3$ by zero extension outside $\tilde{D}_2$ so that the extended function, denoted by $\tilde{G}$, is in $H^2(\mathbb{R}^3)$. Then by Lemma \ref{lemma:agmonR^3}, we have 
\begin{equation} \begin{aligned}
    \Lnorm{\tilde{G}}{\infty(\mathbb{R}^3)} 
    & \leq (1+\sqrt{2}) 2 \sqrt{\pi} \norm{\tilde{G}}_{H^1(\mathbb{R}^3)}^{1/2} \norm{\tilde{G}}_{H^2(\mathbb{R}^3)}^{1/2} 
    \\ & = (1+\sqrt{2}) 2 \sqrt{\pi} \norm{G}_{H^1(\tilde{D}_2)}^{1/2} \norm{G}_{H^2(\tilde{D}_2)}^{1/2}
\end{aligned} \end{equation}
Then we have
\begin{equation} \begin{aligned}
    \Lnorm{\frac{v_r}{r}}{\infty(D)} &\leq \Lnorm{\tilde{G}}{\infty(\mathbb{R}^3)} 
    \\ & \leq (1+\sqrt{2}) 2 \sqrt{\pi} \norm{G}_{H^1(\tilde{D}_2)}^{1/2} \norm{G}_{H^2(\tilde{D}_2)}^{1/2}  
    \\ & \leq (1+\sqrt{2}) 2 \sqrt{\pi} (536)^{1/4}(57452(1+5/\pi^2) + 60)^{1/4} \Lnorm{\nabla \frac{v_r}{r}}{2(D)}^{1/2} \Lnorm{\nabla^2 \frac{v_r}{r}}{2(D)}^{1/2}.
\end{aligned} \end{equation}
\end{proof}

\section{Proof of Theorem \ref{thm:main}}
We will start with two energy estimates on $\Omega$ and $V$. The local existence of solution to $\eqref{eq:ASNS}$ under the conditions of Theorem \ref{thm:main} is known \cite{mitrea2009nonlinear}. There exists a strong solution $v$ to $\eqref{eq:ASNS}$ in $D \times [0,T^*]$ where $T^*$ may be small, and we can use $V$ and $\Omega$ as test functions.

\subsection{Energy estimate for $V$ and $\Omega$} We apply the $L^4$ energy estimate to \eqref{eq:VOmega}. We multiply \eqref{eq:VOmega} by $V^3$ and integrate in space and time to get
\begin{equation} \begin{aligned}
 &\frac{1}{4} \ddt \Lnorm{V^2}{2}^2 + \frac{3}{4} \int_D |\nabla V^2|^2 \, dx  + \frac{3}{4} \int_D \frac{1}{r^2} V^4 \, dx \\ &= \int_D \nabla \cdot(\nabla V (V^3)) \, dx - \frac{3}{2} \int_D \frac{v_r}{r} V^4 \, dx + \frac{1}{4} \int_0^1 V^4(1,x_3,t) \, dx_3.
 \label{eq:ES_V_1}
\end{aligned} \end{equation}
We note that we have an extra boundary term $\frac{1}{4} \int_0^1 V^4(1,x_3,t) \, dx_3$ on the right but this can be handled. First due to the boundary conditions of $V$, we have $\partial_r V(1,x_3,t) = -\frac{3}{2}V(1,x_3,t)$. Then we have
\begin{equation} \begin{aligned}
\int_D \nabla \cdot(\nabla V (V^3)) \, dx = \int_0^1 \partial_r V(1,x_3,t) (V^3(1,x_3,t)) \, dx_3 = -\frac{3}{2}\int_0^1  V^4(1,x_3,t) \, dx_3.
\label{eq:ES_V_2}
\end{aligned} \end{equation}
Hence the boundary term $\frac{1}{4} \int_0^1 V^4(1,x_3,t) \, dx_3$ can be absorbed, the only term on the right that we need to control is $- \frac{3}{2} \int_D \frac{v_r}{r} V^4 \, dx$.


Next, we apply the Lemma \ref{lemma:agmon_inequality} and Lemma \ref{lemma1} to obtain
\begin{equation} \begin{aligned}
\Lnorm{\frac{v_r}{r}}{\infty} &\leq C_1 \Lnorm{\nabla \frac{v_r}{r}}{2}^{1/2} \Lnorm{\nabla^2 \frac{v_r}{r}}{2}^{1/2} \\
&\leq C_1 C_3^{1/2} \Lnorm{\Omega}{2}^{1/2}\Lnorm{\partial_{x_3}\Omega}{2}^{1/2}.
\end{aligned} \end{equation}
Also using H\"{o}lder's inequality, we obtain
\begin{equation} \begin{aligned}
\int_0^1 \int_0^1 \frac{v_\theta^4}{r^2}\,rdrdx_3 &= \int_0^1 \int_0^1 \left(\frac{v_\theta^4}{r^3}\right)^{3/4} \left( v_\theta^4r^5\right)^{1/4} \, drdx_3 \\
& \leq \left( \int_0^1 \int_0^1  \frac{v_\theta^4}{r^3} \, drdx_3 \right)^{3/4} \left( \int_0^1 \int_0^1  v_\theta^4 r^5 \, drdx_3 \right)^{1/4}
\\ & = \left( \int_0^1 \int_0^1  \frac{v_\theta^4}{r^4} \, rdrdx_3 \right)^{3/4} \left( \int_0^1 \int_0^1  v_\theta^4 r^4 \, rdrdx_3 \right)^{1/4},
\end{aligned} \end{equation}
that is 
\begin{equation} \begin{aligned}
\int_D V^4 \,dx \leq \Lnorm{r^{-1}V^2}{2}^{3/2} \Lnorm{\Gamma}{4}.
\end{aligned} \end{equation}
Then
\begin{equation} \begin{aligned}
- \frac{3}{2} \int_D \frac{v_r}{r} V^4 \, dx 
&\leq \frac{3}{2}C_1 C_3^{1/2} \Lnorm{\Omega}{2}^{1/2}\Lnorm{\partial_{x_3}\Omega}{2}^{1/2} \int_D V^4 \, dx \\
& \leq \frac{3}{2}C_1 C_3^{1/2} \Lnorm{\Omega}{2}^{1/2} \Lnorm{\partial_{x_3}\Omega}{2}^{1/2} \Lnorm{r^{-1}V^2}{2}^{3/2} \Lnorm{\Gamma}{4} \\
& \leq \frac{3}{2}C_1 C_3^{1/2} \Lnorm{\Omega}{2}^{1/2}  \Lnorm{\Gamma}{4} \left( \frac{1}{4} \Lnorm{\partial_{x_3}\Omega}{2}^{2} + \frac{3}{4} \Lnorm{r^{-1}V^2}{2}^{2} \right) \\
& \leq \frac{9}{8}C_1 C_3^{1/2}\Lnorm{\Omega}{2}^{1/2}  \Lnorm{\Gamma}{4} \left(  \Lnorm{\partial_{x_3}\Omega}{2}^{2} +  \Lnorm{r^{-1}V^2}{2}^{2} \right)
\label{eq:ES_V_3}
\end{aligned} \end{equation}
where Young's inequality is used from the second to the third line.

Hence putting \eqref{eq:ES_V_2} and \eqref{eq:ES_V_3} into \eqref{eq:ES_V_1}, we get
\begin{equation} \begin{aligned}
 &\frac{1}{4} \ddt \Lnorm{V^2}{2}^2 + \frac{3}{4} \int_D |\nabla V^2|^2 \, dx + \frac{3}{4} \int_D \frac{1}{r^2} V^4 \, dx + \frac{5}{4}\int_0^1  V^4(1,x_3,t) \, dx_3
 \\ & \leq \frac{9}{8}C_1 C_3^{1/2} \Lnorm{\Omega}{2}^{1/2}  \Lnorm{\Gamma}{4} \left( \Lnorm{\partial_{x_3}\Omega}{2}^{2} + \Lnorm{r^{-1}V^2}{2}^{2} \right).
 \label{eq:ES_V_4}
\end{aligned} \end{equation}

Using $\Omega$ as a test function in \eqref{eq:VOmega}, we obtain
\begin{equation} \begin{aligned}
&\frac{1}{2} \ddt \Lnorm{\Omega}{2}^2 + \Lnorm{\nabla \Omega}{2}^2 + \int_0^1 \Omega^2(0,x_3,t)\,dx_3 \\ &= \int_D \frac{\partial_{x_3}V^2}{r} \Omega \, dx = - \int_D \partial_{x_3}\Omega \frac{V^2}{r} \, dx \leq \frac{1}{2} \Lnorm{\partial_{x_3}\Omega}{2}^2 + \frac{1}{2} \Lnorm{r^{-1}V^2}{2}^2,
\label{eq:ES_Omega_1}
\end{aligned} \end{equation}
where we integrate by parts and use the boundary condition of $\Omega$ in the last step.

\subsection{Boundedness of $\norm{V^2(t)}_{L^2}^2$ and $\norm{\Omega(t)}_{L^2}^2$ under the smallness condition \eqref{eq:smallness_condition}} Adding \eqref{eq:ES_V_4} and \eqref{eq:ES_Omega_1} gives
\begin{equation} \begin{aligned}
&\frac{1}{4} \ddt \Lnorm{V^2}{2}^2 + \frac{1}{2} \ddt \Lnorm{\Omega}{2}^2 + \frac{3}{4} \int_D |\nabla V^2|^2 \, dx + \Lnorm{\nabla \Omega}{2}^2 \\& + \frac{5}{4}\int_0^1  V^4(1,x_3,t) \, dx_3 + \int_0^1 \Omega^2(0,x_3,t)\,dx_3 +  \frac{3}{4} \Lnorm{r^{-1}V^2}{2}^2   \\
&\leq \frac{1}{2} \Lnorm{\partial_{x_3}\Omega}{2}^2 + \frac{1}{2} \Lnorm{r^{-1}V^2}{2}^2 + \frac{9C_1 C_3^{1/2}}{8} \Lnorm{\Omega}{2}^{1/2}  \Lnorm{\Gamma}{4} ( \Lnorm{\partial_{x_3}\Omega}{2}^{2} + \Lnorm{r^{-1}V^2}{2}^{2})\\
&\leq \frac{1}{2} \Lnorm{\partial_{x_3}\Omega}{2}^2 + \frac{1}{2} \Lnorm{r^{-1}V^2}{2}^2 + \frac{9C_1 C_3^{1/2}}{8} \Lnorm{\Omega}{2}^{1/2}  \Lnorm{\Gamma_0}{4} ( \Lnorm{\partial_{x_3}\Omega}{2}^{2} + \Lnorm{r^{-1}V^2}{2}^{2}).
\label{eq:energy_estimate_VOmega}
\end{aligned} \end{equation}
Since the solution is strong, there exists $T^{**}\leq T^*$ such that $ \Lnorm{\Omega(t)}{2}^{1/2} \leq 2 \Lnorm{\Omega_0}{2}^{1/2}$ for $t\in (0,T^{**})$. If we can impose the following smallness condition
\begin{equation} \begin{aligned}
 \frac{9C_1 C_3^{1/2}}{4} \left(\frac12 \norm{V_0^2}_{L^2}^2 +  \Lnorm{\Omega_0}{2}^{2} \right)^{1/4} \Lnorm{\Gamma_0}{4} \leq \frac{1}{4}.
 \label{eq:same_smallnesss_condition}
\end{aligned} \end{equation}
Then from \eqref{eq:energy_estimate_VOmega} we have
\begin{equation} \begin{aligned}
&\frac{1}{4} \ddt \Lnorm{V^2}{2}^2 + \frac{1}{2} \ddt \Lnorm{\Omega}{2}^2 + \frac{3}{4} \int_D |\nabla V^2|^2 \, dx + \Lnorm{\nabla \Omega}{2}^2 \\& +  \frac{5}{4}\int_0^1  V^4(1,x_3,t) \, dx_3 + \int_0^1 \Omega^2(0,x_3,t)\,dx_3 +  \frac{3}{4} \Lnorm{r^{-1}V^2}{2}^2   
\\ & \leq  \frac{1}{2} \Lnorm{\partial_{x_3}\Omega}{2}^2 + \frac{1}{2} \Lnorm{r^{-1}V^2}{2}^2 + \frac{9C_1 C_3^{1/2}}{8} \Lnorm{\Omega}{2}^{1/2}  \Lnorm{\Gamma_0}{4} ( \Lnorm{\partial_{x_3}\Omega}{2}^{2} + \Lnorm{r^{-1}V^2}{2}^{2}) 
\\ & \leq  \frac{1}{2} \Lnorm{\partial_{x_3}\Omega}{2}^2 + \frac{1}{2} \Lnorm{r^{-1}V^2}{2}^2 + \frac{9C_1 C_3^{1/2}}{4}( \Lnorm{\Omega_0}{2}^{2} )^{1/4}  \Lnorm{\Gamma_0}{4} ( \Lnorm{\partial_{x_3}\Omega}{2}^{2} + \Lnorm{r^{-1}V^2}{2}^{2}) 
\\ & \leq  \frac{1}{2} \Lnorm{\partial_{x_3}\Omega}{2}^2 + \frac{1}{2} \Lnorm{r^{-1}V^2}{2}^2 + \frac{9C_1 C_3^{1/2}}{4} \left( \frac12 \norm{V_0^2}_{L^2}^{2} + \Lnorm{\Omega_0}{2}^{2}  \right)^{1/4} \Lnorm{\Gamma_0}{4}  ( \Lnorm{\partial_{x_3}\Omega}{2}^{2} + \Lnorm{r^{-1}V^2}{2}^{2}) 
\label{eq:after_smallness}
\end{aligned} \end{equation}

Then 
\begin{equation} \begin{aligned}
\label{eq:final_energy_estimate}
 \frac{1}{4} \ddt \Lnorm{V^2(t)}{2}^2 + \frac{1}{2} \ddt \Lnorm{\Omega(t)}{2}^2\leq 0 \quad \text{for $t\in (0,T^{**})$.}
\end{aligned} \end{equation}
The smallness condition \eqref{eq:same_smallnesss_condition} is global since from \eqref{eq:after_smallness}, we have for $t \in [0,T^{**})$
\begin{equation}
    \frac14 \norm{V^2(t)}_{L^2}^2 + \frac12 \norm{\Omega(t)}_{L^2}^2 \leq \frac14 \norm{V_0^2}_{L^2}^2 + \frac12 \norm{\Omega_0}_{L^2}^2.
\end{equation}
Then we have
\begin{equation}
    \norm{\Omega(T^{**})}_{L^2}^{1/2} \leq \left( \frac12 \norm{V_0^2}_{L^2}^2 + \norm{\Omega_0}_{L^2}^2 \right)^{1/4}
    \label{eq:smallness_after_time}
\end{equation}
Performing the same analysis again, we will have \eqref{eq:energy_estimate_VOmega} where $t \in [T^{**}, T^{***})$ for some $T^{***}$. For the last term in \eqref{eq:energy_estimate_VOmega}, after imposing the same smallness condition \eqref{eq:same_smallnesss_condition} and using \eqref{eq:smallness_after_time} we get
\begin{equation} \begin{aligned}
    & \frac{9C_1 C_3^{1/2}}{8} \Lnorm{\Omega}{2}^{1/2}  \Lnorm{\Gamma_0}{4} ( \Lnorm{\partial_{x_3}\Omega}{2}^{2} + \Lnorm{r^{-1}V^2}{2}^{2}) 
    \\ &\leq  \frac{9C_1 C_3^{1/2}}{4} \Lnorm{\Omega(T^{**})}{2}^{1/2}  \Lnorm{\Gamma_0}{4} ( \Lnorm{\partial_{x_3}\Omega}{2}^{2} + \Lnorm{r^{-1}V^2}{2}^{2}) 
    \\ &\leq \frac{1}{4}  ( \Lnorm{\partial_{x_3}\Omega}{2}^{2} + \Lnorm{r^{-1}V^2}{2}^{2}) 
\end{aligned} \end{equation}
thus the same conclusion \eqref{eq:final_energy_estimate} holds for $t \in [T^{**},T^{***})$. The solution can be continued again. Hence the smallness condition \eqref{eq:same_smallnesss_condition} is a global condition.

Concluding by the continuation principle, $||\Omega(t)||^2_{L^2}$ can be continued for any $t>0$. This shows that $||\Omega(t)||_{L^2} < \infty$ for any $t \geq 0$. 

\subsection{$L^4$ bound for $b$} Since the normal component of $b$ vanishes, from the inequality $\norm{\nabla b}_{L^2}^2 \leq C\norm{ \operatorname{div} \, b}_{L^2}^2 + C \norm{\operatorname{curl} \, b}_{L^2}^2$ \cite{von1992estimating}, we see that
\begin{equation} \begin{aligned}
\norm{\nabla b}_{L^2}^2 \leq C \norm{\operatorname{curl} \, b}_{L^2}^2 = C \norm{\omega_\theta(t)}_{L^2}^2 \leq C \norm{\Omega(t)}_{L^2}^2 < \infty.
\end{aligned}
\end{equation}
On $(0,T^*]$, by the energy inequality we have $\| b \|_{L^2} \leq \|v_0\|_{L^2} < \infty$.
Hence for $t\in(0,T^*]$
\begin{equation} \begin{aligned}
v_re_r + v_3e_3 = b \in L^4.
\end{aligned}
\end{equation}
By the continuation principle, this implies that for any $t \geq 0$, $v_r \in L^4$ and $v_3 \in L^4$.

\subsection{$L^4$ bound for $v_\theta$} 
We multiply $v_\theta^3$ to the equation of $v_\theta$ in \eqref{eq:ASNS}. The Laplace term can be computed using the boundary condition in the following way
\begin{equation}\begin{aligned}
    \int_D v_\theta^3 \, \Delta v_\theta \, dx & = \int_D \nabla\cdot(v_\theta^3 \, \nabla v_\theta ) - \frac34 \int_D |\nabla( v_\theta^2)|^2 \, dx 
    \\ & = - \int_{\partial^V D} v_\theta^4 \, dx_3 - \frac34 \int_D |\nabla( v_\theta^2)|^2 \, dx .
\end{aligned}\end{equation}
Computing the other terms we have
\begin{equation}
    \frac{1}{4}\ddt \norm{v_\theta}_{L^4}^4 + \frac{3}{4} \norm{\nabla v_\theta^2}_{L^2}^2 + \norm{r^{-1} v_\theta^2}_{L^2}^2 + \int_0^1 v_\theta^4(1,x_3,t) \, dx_3  \leq  \int_D \frac{|v_r|}{r} v_\theta^4 \, dx.
\end{equation}
Using H\"{o}lder's inequality, Cauchy–Schwarz inequality and the fact that $rv_\theta = \Gamma$ we have,
\begin{equation} \begin{aligned}
\int_0^1 \int_0^1 \frac{|v_r|}{r} v_\theta^4 \, r dr dx_3 &=  \int_0^1 \int_0^1 \frac{|v_r|}{r} \frac{v_\theta^2}{r} v_\theta^2 \, r^2 dr dx_3 \\
& \leq \norm{\Gamma_0}_{L^\infty} \int_0^1 \int_0^1 \frac{|v_r|}{r} \frac{v_\theta^2}{r} v_\theta \, r dr dx_3 \\
& \leq \norm{\Gamma_0}_{L^\infty} \left[ \int_D \left( \frac{v_r}{r} \right)^4 dx \right]^\frac{1}{4}   \left[ \int_D \left( \frac{v_\theta^2}{r} \right)^2 dx \right]^\frac{1}{2}   \left[ \int_D v_\theta^4 dx \right]^\frac{1}{4} \\
& \leq \frac{1}{2}  \left[ \int_D \left( \frac{v_\theta^2}{r} \right)^2 dx \right] + \frac{1}{2} \norm{\Gamma_0}_{L^\infty}^2 \left[ \int_D \left( \frac{v_r}{r} \right)^4 \, dx \right]^\frac{1}{2}  \left[ \int_D v_\theta^4 \, dx \right]^\frac{1}{2} \\
& \leq \frac{1}{2}  \norm{r^{-1} v_\theta^2}_{L^2}^2 + \frac{1}{2}  \norm{\Gamma_0}_{L^\infty}^2 \norm{\frac{v_r}{r}}_{L^4}^2 \norm{v_\theta}_{L^4}^2.
\end{aligned}
\end{equation}
Then it follows from Lemma \ref{lemma1}, Sobolev inequality and Poincar\'e inequality that
\begin{equation} \begin{aligned}
 &\frac{1}{4} \ddt \norm{v_\theta}_{L^4}^4 \\ &\leq  \frac{1}{2} \norm{\Gamma_0}_{L^\infty}^2 \norm{\frac{v_r}{r}}_{L^4}^2 \norm{v_\theta}_{L^4}^2
 \\ & \leq \frac{C}{2} \norm{\Gamma_0}_{L^\infty}^2 \norm{\frac{v_r}{r}}_{L^6}^2 
 \norm{v_\theta}_{L^4}^2 
 \\ & \leq \frac{C}{2} \norm{\Gamma_0}_{L^\infty}^2 \left( \norm{\frac{v_r}{r}}_{L^2}^2 + \norm{\nabla \frac{v_r}{r}}_{L^2}^2  \right) \norm{v_\theta}_{L^4}^2 
 \\ & \leq C\norm{\Gamma_0}_{L^\infty}^2 \norm{\nabla \frac{v_r}{r}}_{L^2}^2  \norm{v_\theta}_{L^4}^2 
 \\ &\leq C \norm{\Gamma_0}_{L^\infty}^2 \norm{\Omega(t)}_{L^2}^2 \norm{v_\theta}_{L^4}^2 
 \\ & \leq C \norm{v_\theta}_{L^4}^2.
\end{aligned} \end{equation}
Let $f = \norm{v_\theta}_{L^4}^4$, we have a differential inequality 
\begin{equation}
\begin{aligned}
    \ddt f \leq 4C \sqrt{f}.
\end{aligned}
\end{equation}
Solving for this inequality we obtain,
\begin{equation}
    \sqrt{f(t)} \leq 2Ct + \sqrt{f(0)} \quad \forall  \, t\in [0,T],
\end{equation}
That is 
\begin{equation}
    \norm{v_\theta}_{L^4}^2 \leq 2Ct + \norm{v_{\theta,0}}_{L^4}^2 < \infty \quad \forall  \, t\in [0,T].
\end{equation}

\subsection{Completion of the proof} We have shown that $v_r$, $v_\theta$ and $v_3$ are all in $L_T^\infty L_x^4$ for arbitrary $T$. Hence by the Ladyzhenskaya--Prodi--Serrin regularity criterion of \cite{serrin1962interior,prodi1959teorema,ladyzhenskaya1967uniqueness}, $v$ is regular up to any finite time $T$ in the interior of $D$. 

Also it is only possible for a blow-up to occur on $r=0$ by the result in \cite{caffarelli1982partial}. Hence we can rule out such blow up except for the case on $ \{ r=0 \} \cap \partial^H D$. But by the result in \cite{neustupa2018regularity}, for NHL boundary condition, the interior regularity criterion also implies regularity on $\partial^H D$. So $v$ is regular up to any finite time $T$ in $D$.  \qed

\section*{Acknowledgments} We are deeply grateful to Professors Qi S. Zhang, Zijin Li, Xinghong Pan, Xin Yang, Na Zhao and Daoguo Zhou for the helpful discussions and suggestions.

\bibliographystyle{plain}

\end{document}